\documentclass[a4paper,12pt]{article}

\usepackage{amsmath}
\usepackage{amsthm}
\usepackage{amssymb}

\usepackage{latexsym}
\usepackage{array}
\usepackage{enumerate}

\numberwithin{equation}{section}

\newcommand{\CC}{\mathbb{C}}

\newcommand{\RR}{\mathbb{R}}

\newcommand{\ZZ}{\mathbb{Z}}
\newcommand{\D}{\mathcal{D}}

\newcommand{\LL}{\mathcal{L}}
\newcommand{\M}{\mathcal{M}}
\newcommand{\N}{\mathcal{N}}

\newcommand{\Lin}{{\rm Lin}}

\renewcommand{\dim}{{\rm dim}}

\newcommand{\Vol}{{\rm Vol}}

\newcommand{\e}{\varepsilon}

\newcommand{\Spec}{{\rm Spec}}

\newcommand{\supp}{{\rm supp}}

\newcommand{\Int}{{\rm Int}}

\newcommand{\Db}{{\bf D}^{b}}
\newcommand{\Dbc}{{\bf D}_{c}^{b}}

\newcommand{\tl}[1]{\widetilde{#1}}

\newcommand{\simto}{\overset{\sim}{\longrightarrow}}

\newcommand{\dsum}{\displaystyle \sum}

\newtheorem{definition}{Definition}[section]
\newtheorem{theorem}[definition]{Theorem}
\newtheorem{proposition}[definition]{Proposition}
\newtheorem{lemma}[definition]{Lemma}
\newtheorem{corollary}[definition]{Corollary}

\newtheorem{remark}[definition]{Remark}

\title{On vanishing theorems for local systems 
associated to Laurent polynomials 
}

\author{Alexander ESTEROV 
\footnote{ National Research University 
Higher School of Economics \newline Faculty 
of Mathematics NRU HSE, 7 Vavilova 117312 
Moscow, Russia. 
E-mail: aesterov@hse.ru \newline 
Partially supported by RFBR 
grant 16-01-00409.} 
and Kiyoshi TAKEUCHI 
\footnote{Institute of Mathematics, University  of 
Tsukuba, 1-1-1, Tennodai, 
Tsukuba, Ibaraki, 305-8571, Japan. 
E-mail: takemicro@nifty.com 
TEL: 029-853-7650(Japan) FAX: 
029-853-6501(Japan) } 
}

\date{}

\begin{document}

\maketitle

\begin{abstract}
We prove some vanishing theorems 
for the cohomology groups of 
local systems associated to 
Laurent polynomials. 
In particular, we extend one of 
the results of 
Gelfand-Kapranov-Zelevinsky 
\cite{G-K-Z-2} to various 
directions. In the course of the proof, 
some properties of vanishing cycles of perverse 
sheaves and twisted Morse theory 
will be used. 
\end{abstract}


\section{Introduction}\label{sec:1}

The study of the cohomology groups of 
local systems is an important subject in 
algebraic geometry, hyperplane arrangements, 
topology and hypergeometric 
functions of several variables. 
Many mathematicians are interested 
in the conditions for 
which we have their concentrations  
in the middle degrees (for a review of this 
subject, see for example \cite[Section 6.4]{Dimca}). 
Here let us consider this problem in the 
following situation. 
Let $B=\{ b(1), b(2), \ldots , 
b(N)\} \subset \ZZ^{n-1}$ be a finite subset 
of the lattice $\ZZ^{n-1}$. 
Assume that the affine lattice generated by 
$B$ in $\ZZ^{n-1}$ coincides with $\ZZ^{n-1}$. 
For $z=(z_1, \ldots, z_N) \in \CC^N$ 
we consider Laurent polynomials $P(x)$ 
on the algebraic torus 
$T_0=( \CC^*)^{n-1}$ defined by 
$P(x)=\sum_{j=1}^N z_j x^{b(j)}$ 
($x=(x_1, \ldots, x_{n-1}) 
\in T_0=( \CC^*)^{n-1}$). 
Then for $c=(c_1, \ldots, c_n) \in \CC^n$ 
we obtain a possibly multi-valued function 
$P(x)^{-c_n} 
x_1^{c_1-1} \cdots x_{n-1}^{c_{n-1}-1}$ on 
$W=T_0 \setminus P^{-1}(0)$. 
It generates the 
rank one local system 
\begin{equation} 
\LL = \CC_{W} 
P(x)^{-c_n} 
x_1^{c_1-1} \cdots x_{n-1}^{c_{n-1}-1} 
\end{equation}
on $W$.  Under the nonresonance 
condition (see Definition \ref{NRC}) on 
$c \in \CC^n$, 
Gelfand-Kapranov-Zelevinsky 
\cite{G-K-Z-2} proved that 
we have the concentration
\begin{equation} 
H^j(W ; \LL ) \simeq 
0 \qquad (j \not= n-1)
\end{equation}
for non-degenerate Laurent polynomials $P(x)$. 
This result was obtained as a byproduct 
of their study on the integral 
representations of $A$-hypergeometric 
functions in \cite{G-K-Z-2}. Since 
their proof of this concentration 
heavily relies on 
the framework of the $\D$-module 
theory, it is desirable to 
prove it more directly. 
In this paper, by applying the 
twisted Morse theory to perverse sheaves 
we extend the result of 
Gelfand-Kapranov-Zelevinsky 
to various directions. 
First in Theorem \ref{VTM} 
we relax the non-degeneracy 
condition on $P(x)$ by replacing it 
with a weaker one (see Definition 
\ref{WND}). We thus  
extend the result of \cite{G-K-Z-2} 
to the case where the hypersurface $P^{-1}(0) 
\subset T_0$ may have isolated 
singular points in $T_0$. 
In fact, in Theorem \ref{VTM} 
we relax also the condition that 
$B$ generates $\ZZ^{n-1}$ to a weaker one that 
the dimension of the convex 
hull $\Delta \subset \RR^{n-1}$ of 
$B$ in $\RR^{n-1}$ is $n-1$. In Theorem \ref{MVTM} 
we extend these results 
to more general local systems 
associated to several Laurent polynomials. 
Namely we obtain a vanishing theorem for 
arrangements of toric hypersurfaces 
with isolated singular points. Our proofs 
of Theorems \ref{VTM} and \ref{MVTM} 
are very natural and obtained by taking 
(possibly singular) ``minimal" toric 
compactifications of $T_0$. In order to work 
on such singular varieties, we use 
our previous idea in the proof of 
\cite[Lemma 4.2]{E-T-2}. 
See Section \ref{sec:3} for the details. 
Moreover in Theorem \ref{NTM} 
(assuming the non-degeneracy of 
Gelfand-Kapranov-Zelevinsky \cite{G-K-Z-2} 
for Laurent polynomials) we relax 
the nonresonance condition of $c \in \CC^n$  
in Theorem \ref{MVTM} by replacing it with the 
much weaker one $c \notin \ZZ^n$. 
To prove Theorem \ref{NTM}, we first perturb  
Laurent polynomials by multiplying 
monomials. Then we apply the 
twisted Morse theory to the 
real-valued functions 
associated to them 
by using some standard properties 
of vanishing cycles of perverse 
sheaves. See Sections \ref{sec:4} 
and \ref{sec:5} for the details. 
In the course of the proof of Theorem \ref{NTM}, 
we obtain also the following result which 
might be of independent interest. 
Let $Q_1, \ldots, Q_l$ be 
Laurent polynomials on $T=( \CC^*)^{n}$ 
and for $1 \leq i \leq l$ denote by 
$\Delta_i \subset \RR^{n}$ the Newton 
polytope $NP(Q_i)$ of $Q_i$. 
Set $\Delta = \Delta_1 + \cdots + \Delta_l$.

\begin{theorem}\label{NNVTM} 
Let $\LL$ be a non-trivial local system 
of rank one on $T=( \CC^*)^{n}$. 
Assume that for any $1 \leq i \leq l$ 
we have $\dim \Delta_i =n$ and the subvariety 
\begin{equation} 
Z_i = \{ x \in T \ | \ 
Q_1(x)= \cdots = Q_i(x)=0 \} \subset T 
\end{equation} 
of $T$ is a non-degenerate complete 
intersection. 
Then for any $1 \leq i \leq l$ we have the 
concentration 
\begin{equation} 
H^j(Z_i ; \LL ) \simeq 
0 \qquad (j \not= n-i). 
\end{equation} 
Moreover we have 
\begin{equation}
\dim H^{n-i} (Z_i ; \LL ) =  
\dsum_{\begin{subarray}{c} 
m_1,\ldots,m_i \geq 1\\ 
m_1+\cdots +m_i=n 
\end{subarray}}\Vol_{\ZZ}(
\underbrace{\Delta_1,
\ldots,\Delta_1}_{\text{
$m_1$-times}},\ldots, 
\underbrace{\Delta_i,
\ldots,\Delta_i}_{\text{$m_i$-times}}),  
\end{equation}
where $\Vol_{\ZZ}(\underbrace{\Delta_1,
\ldots,\Delta_1}_{\text{$m_1$-times}},
\ldots,\underbrace{\Delta_i,\ldots,
\Delta_i}_{\text{$m_i$-times}})\in \ZZ$ 
is the normalized $n$-dimensional mixed volume 
with respect to the lattice $\ZZ^n 
\subset \RR^n$
\end{theorem}
Note that this result can be 
considered as a refinement of 
the classical Bernstein-Khovanskii-Kushnirenko 
theorem (see \cite{Khovanskii}). 
On the other hand, Matusevich-Miller-Walther 
\cite{M-M-W} and Saito-Sturmfels-Takayama 
\cite{S-S-T} studied the condition on 
the parameter vector $c \in \CC^n$ 
for which the corresponding local 
system of $A$-hypergeometric 
functions is non-rank-jumping. They also 
relaxed the nonresonance condition of 
$c \in \CC^n$. It would be an 
interesting problem to study the 
relationship between Theorem \ref{NTM} 
and their results.

\bigskip
\noindent{\bf Acknowledgement:} 
We express our hearty gratitude to 
Professor N. Takayama for drawing 
our attention to this problem. 
Moreover some discussions with 
Professor M. Yoshinaga were very 
useful during the preparation 
of this paper. We also thank the referee 
for giving us many valuable suggestions.

\section{Preliminary results}\label{sec:2}

In this section, we recall basic 
notions and results which will be used 
in this paper. In this paper, we essentially 
follow the terminology of 
\cite{Dimca}, \cite{H-T-T} etc. 
For example, for a topological 
space $X$ we denote by $\Db(X)$ the 
derived category whose objects are 
bounded complexes of sheaves 
of $\CC_X$-modules on $X$. 
Denote by $\Dbc(X)$ the full 
subcategory of $\Db(X)$ consisting of 
constructible objects. 
Let $\Delta \subset \RR^n$ be a lattice polytope 
in $\RR^n$. For an element $u \in \RR^n$ of 
(the dual vector space of) $\RR^n$ we define the 
supporting face $\gamma_u \prec  \Delta$ 
of $u$ in $ \Delta$ by 
\begin{equation}
\gamma_u = \left\{ v \in \Delta \ | \ 
\langle u , v \rangle 
= 
\min_{w \in \Delta } 
\langle u ,w \rangle \right\}, 
\end{equation}
where for $u=(u_1,\ldots,u_n)$ 
and $v=(v_1,\ldots, v_n)$ we set 
$\langle u,v\rangle =\sum_{i=1}^n u_iv_i$. 
For a face $\gamma$ of $\Delta$ set 
\begin{equation}
\sigma (\gamma) = \overline{ \{ u \in \RR^n \ | \ 
\gamma_u = \gamma   \} } \subset \RR^n . 
\end{equation}
\noindent Then  $\sigma (\gamma )$ 
is an $(n- \dim \gamma )$-dimensional 
rational convex polyhedral 
cone in $\RR^n$. Moreover 
the family $\{ \sigma (\gamma ) \ | \ 
\gamma \prec  \Delta \}$ of cones in $\RR^n$ 
thus obtained is a subdivision of $\RR^n$. 
We call it the dual subdivision of $\RR^n$ by 
$\Delta$. If $\dim \Delta =n$ it 
satisfies the axiom 
of fans (see \cite{Fulton} and 
\cite{Oda} etc.). We call it the dual fan of 
$\Delta$. 

Let $\Delta_1, \ldots, \Delta_p  
\subset \RR^n$ be lattice polytopes 
in $\RR^n$ and $\Delta = 
\Delta_1 + \cdots + \Delta_p  
\subset \RR^n$ their Minkowski sum. 
For a face $\gamma \prec \Delta$ of 
$\Delta$, by taking a point $u \in \RR^n$ 
in the relative interior of its dual cone $\sigma (\gamma)$ 
we define the supporting face 
$\gamma_i \prec \Delta_i$ of $u$ in $\Delta_i$. 
Then it is easy to see that 
$\gamma = \gamma_1 + \cdots + \gamma_p$. 
Now we recall Bernstein-Khovanskii-Kushnirenko's 
theorem \cite{Khovanskii}.

\begin{definition}
Let $g(x)=\sum_{v \in \ZZ^n} c_vx^v$ be a 
Laurent polynomial on the algebraic torus 
$T=(\CC^*)^n$ ($c_v\in \CC$). 
\begin{enumerate}
\item We call the convex hull of 
$\supp(g):=\{v\in \ZZ^n \ | \ c_v\neq 0\} 
\subset \ZZ^n \subset \RR^n$ in $\RR^n$ the 
Newton polytope of $g$ and denote it by $NP(g)$.
\item For a face $\gamma \prec NP(g)$ of $NP(g)$, 
we define the $\gamma$-part 
$g^{\gamma}$ of $g$ by 
$g^{\gamma}(x):=\sum_{v \in \gamma} c_vx^v$. 
\end{enumerate}
\end{definition}

\begin{definition}\label{non-deg} (see \cite{khov0}, \cite{Oka} etc.) 
Let $g_1, g_2, \ldots , g_p$ be 
Laurent polynomials on $T=(\CC^*)^n$. 
Set $\Delta_i=NP(g_i)$ $(i=1,\ldots, p)$ and 
$\Delta = \Delta_1 + \cdots + \Delta_p$. 
Then we say that the subvariety 
$Z=\{ x\in T=(\CC^*)^n \ | \ g_1(x)=g_2(x)= 
\cdots =g_p(x)=0 \}$ of $T=(\CC^*)^n$ is a 
non-degenerate complete intersection 
if for any face $\gamma \prec \Delta$ of 
$\Delta$ the $p$-form $dg_1^{\gamma_1} \wedge 
dg_2^{\gamma_2} \wedge 
\cdots \wedge dg_p^{\gamma_p}$ does not vanish 
on $\{ x\in T=(\CC^*)^n  \ | \ 
g_1^{\gamma_1}(x)= \cdots =
g_p^{\gamma_p}(x)=0 \}$.
\end{definition}

\begin{definition}\label{rem:2-13}
Let $\Delta_1,\ldots,\Delta_n$ 
be lattice 
polytopes in $\RR^n$. Then 
their normalized $n$-dimensional 
mixed volume 
$\Vol_{\ZZ}( \Delta_1,\ldots,\Delta_n) 
\in \ZZ$ is defined by the formula 
\begin{equation}
\Vol_{\ZZ}( \Delta_1, \ldots , \Delta_n)=
\frac{1}{n!} 
\dsum_{k=1}^n (-1)^{n-k} 
\sum_{\begin{subarray}{c}I\subset 
\{1,\ldots,n\}\\ \sharp I=k\end{subarray}}
\Vol_{\ZZ}\left(
\dsum_{i\in I} \Delta_i \right)
\end{equation}
where $\Vol_{\ZZ}(\ \cdot\ )
= n! \Vol (\ \cdot\ ) \in \ZZ$ is 
the normalized $n$-dimensional volume 
with respect to the lattice $\ZZ^n 
\subset \RR^n$.
\end{definition}

\begin{theorem}\label{BKK} 
( \cite{Khovanskii} )\label{thm:2-14}
Let $g_1, g_2, \ldots , g_p$ be 
Laurent polynomials on $T=(\CC^*)^n$. 
Assume that the subvariety $Z
=\{ x\in T=(\CC^*)^n \ | \ g_1(x)=g_2(x)= 
\cdots =g_p(x)=0 \}$ of $T=(\CC^*)^n$ is a 
non-degenerate complete intersection. 
Set $\Delta_i=NP(g_i)$ $(i=1,\ldots, p)$. Then we have
\begin{equation}
\chi(Z)=(-1)^{n-p} 
\dsum_{\begin{subarray}{c} 
m_1,\ldots,m_p \geq 1\\ m_1+\cdots +m_p=n 
\end{subarray}}\Vol_{\ZZ}(
\underbrace{\Delta_1,\ldots,\Delta_1}_{\text{
$m_1$-times}},\ldots, 
\underbrace{\Delta_p,
\ldots,\Delta_p}_{\text{$m_p$-times}}),
\end{equation}
where $\Vol_{\ZZ}(\underbrace{\Delta_1,
\ldots,\Delta_1}_{\text{$m_1$-times}},
\ldots,\underbrace{\Delta_p,\ldots,
\Delta_p}_{\text{$m_p$-times}})\in \ZZ$ 
is the normalized $n$-dimensional mixed volume 
with respect to the lattice $\ZZ^n 
\subset \RR^n$.
\end{theorem}

\section{A vanishing theorem 
for local systems}\label{sec:3}

Let $B=\{ b(1), b(2), \ldots , 
b(N)\} \subset \ZZ^{n-1}$ be a finite subset 
of the lattice $\ZZ^{n-1}$. 
Let $\Delta \subset \RR^{n-1}$ be the convex 
hull of $B$ in $\RR^{n-1}$. 
Assume that $\dim \Delta =n-1$. 
For $z=(z_1, \ldots, z_N) \in \CC^N$ 
we define a Laurent polynomial $P(x)$ 
on $T_0=( \CC^*)^{n-1}$ by 
$P(x)=\sum_{j=1}^N z_j x^{b(j)}$ 
($x=(x_1, \ldots, x_{n-1}) 
\in T_0=( \CC^*)^{n-1}$). 
Then for $c=(c_1, \ldots, c_n) \in \CC^n$ 
the possibly multi-valued function 
$P(x)^{-c_n} 
x_1^{c_1-1} \cdots x_{n-1}^{c_{n-1}-1}$ on 
$W=T_0 \setminus P^{-1}(0)$ generates the 
local system 
\begin{equation} 
\LL = \CC_{W} 
P(x)^{-c_n} 
x_1^{c_1-1} \cdots x_{n-1}^{c_{n-1}-1}. 
\end{equation} 
Set $a(j)=(b(j), 1) \in \ZZ^n$ 
($1 \leq j \leq N$) and $A=\{ a(1), a(2), \ldots , 
a(N)\} \subset \ZZ^{n}$. 
Then $K= \RR_+ A \subset \RR^n$ is an 
$n$-dimensional  
closed convex polyhedral cone in $\RR^n$. 
For a face $\Gamma \prec K$ of $K$ 
let $\Lin (\Gamma) \simeq 
\CC^{\dim \Gamma} \subset \CC^n$ 
be the $\CC$-linear subspace of 
$\CC^n$ generated by $\Gamma$. 

\begin{definition}\label{NRC} 
(Gelfand-Kapranov-Zelevinsky 
\cite[page 262]{G-K-Z-2}) 
We say that the parameter vector 
$c \in \CC^n$ is nonresonant 
(with respect to $A$) if 
for any face $\Gamma \prec K$ of $K$ 
such that $\dim \Gamma =n-1$ 
we have $c \notin \{ \ZZ^n+ 
\Lin (\Gamma ) \}$. 
\end{definition} 
The following definition is essentially 
weaker than the usual (Kouchnirenko's) 
non-degeneracy (see \cite{khov0}, \cite{Oka} etc.).

\begin{definition}\label{WND} 
We say that the Laurent polynomial 
$P(x)= \sum_{j=1}^N z_j x^{b(j)}$ is 
``weakly" non-degenerate if 
for any face $\gamma$ of 
$\Delta$ such that 
$\dim \gamma < \dim \Delta =n-1$ the hypersurface 
\begin{equation} 
\{ x \in T_0=(\CC^*)^{n-1} \ | \ 
P^{\gamma}(x)= \sum_{j: b(j) \in \gamma} 
z_j x^{b(j)}=0 \} \subset T_0 
\end{equation}
is smooth and reduced. 
\end{definition}
Let $\iota : W=T_0 \setminus P^{-1}(0)
 \hookrightarrow T_0$ be the inclusion 
map and set $\M = R \iota_* \LL 
\in \Dbc (T_0)$. Then the following theorem 
generalizes one of the results in 
Gelfand-Kapranov-Zelevinsky \cite{G-K-Z-2} 
to the case where the hypersurface 
$P^{-1}(0) \subset T_0$ may have 
isolated singular points. 

\begin{theorem}\label{VTM} 
Assume that $\dim \Delta =n-1$, 
the parameter vector 
$c \in \CC^n$ is nonresonant and 
the Laurent polynomial $P(x)$ is 
weakly non-degenerate. Then there exists 
an isomorphism 
\begin{equation} 
H^j_c(T_0; \M ) \simeq H^j(T_0; \M ) 
\simeq H^j(W ; \LL ) 
\end{equation} 
for any $j \in \ZZ$. Moreover we have the 
concentration 
\begin{equation} 
H^j(W ; \LL ) \simeq 
0 \qquad (j \not= n-1). 
\end{equation} 
\end{theorem}
\begin{proof} 
Let $\Sigma_0$ be the dual fan of $\Delta$ in $\RR^{n-1}$ 
and $X$ the (possibly singular) toric variety associated to it. 
Then there exists a natural action of $T_0$ on 
$X$ whose orbits are parametrized by the faces of 
$\Delta$. For a face $\gamma$ of $\Delta$ 
denote by $X_{\gamma} \simeq 
(\CC^*)^{\dim \gamma}$ the $T_0$-orbit 
associated to $\gamma$. Note that 
$X_{\Delta} \simeq T_0$ is the unique 
open dense $T_0$-orbit in $X$ and 
its complement $X \setminus X_{\Delta}$ 
is the union of $X_{\gamma}$ for 
$\gamma \prec \Delta$ such that 
$\dim \gamma <n-1$. Let $i : 
X_{\Delta} \simeq T_0 \hookrightarrow X$ be 
the inclusion map. Then by the weak 
non-degeneracy of $P(x)$, the closure 
$S= \overline{i (P^{-1}(0))} 
\subset X$ 
of the hypersurface $i( P^{-1}(0)) \subset i (T_0)$ 
in $X$ intersects $T_0$-orbits $X_{\gamma}$ 
in $X \setminus X_{\Delta}$ transversally. 
Moreover by the nonresonance 
of $c \in \CC^n$, for any 
$\gamma \prec \Delta$ such that $\dim 
\gamma =n-2$ the monodromy of the local 
system $\LL$ around the codimension-one 
$T_0$-orbit $X_{\gamma} \subset X$ in $X$ 
is non-trivial. 
Indeed, let $\gamma \prec \Delta$ 
be such a facet of $\Delta$. 
We denote by $\Gamma$ the facet of the cone 
$K= \RR_+ A$ generated by 
$\gamma \times \{ 1 \}  \subset K$. 
Let $\nu \in \ZZ^{n-1} 
\setminus \{ 0 \}$ be the primitive 
inner conormal vector of the facet 
$\gamma$ of $\Delta \subset 
\RR^{n-1}$ and set 
\begin{equation} 
m= \min_{v \in \Delta} 
\langle \nu, v \rangle = 
\min_{v \in \gamma} 
\langle \nu, v \rangle \in \ZZ. 
\end{equation} 
Then the primitive 
inner conormal vector 
$\widetilde{\nu} \in \ZZ^{n} 
\setminus \{ 0 \}$ of the facet 
$\Gamma$ of $K \subset \RR^{n}$ 
is explicitly given by the formula 
\begin{equation} 
\widetilde{\nu} = 
\left(  \begin{array}{c}
      \nu \\
      -m 
    \end{array}  \right) \in \ZZ^{n} 
\setminus \{ 0 \}. 
\end{equation}
and the condition 
$c=(c_1, \ldots, c_{n-1}, c_n) \notin 
\{ \ZZ^n+ \Lin (\Gamma ) \}$ is 
equivalent to the one 
\begin{equation} 
m( \gamma ):= 
\biggl\langle \nu , \quad 
\left(  \begin{array}{c}
      c_1-1 \\
      \vdots \\
      c_{n-1}-1 
    \end{array}  \right)  \biggr\rangle 
- m \cdot c_n 
\quad  \notin \ZZ. 
\end{equation} 
We can easily see that the 
order of the (multi-valued) function 
$P(x)^{-c_n} 
x_1^{c_1-1} \cdots x_{n-1}^{c_{n-1}-1}$ 
along the codimension-one 
$T_0$-orbit $X_{\gamma} \subset X$ in $X$ 
is equal to $m( \gamma ) \notin \ZZ$. 
Then by constructing suitable distance functions 
as in the proof of 
\cite[Lemma 4.2]{E-T-2}, we can show that 
for the open embedding $i: T_0 \hookrightarrow 
X$ we have 
\begin{equation}
(Ri_* \M )_p \simeq 0 
\qquad \text{for any} \ p \in 
X \setminus i(T_0) 
\end{equation}
as follows. Let us first assume that the point $p \in 
X \setminus i(T_0)$ lies in a $0$-dimensional 
$T_0$-orbit $X_{\gamma}$. Let $U_{\gamma} 
\subset X$ be an $(n-1)$-dimensional affine toric 
variety containing $\{ p \} =X_{\gamma}$ and regard it 
as a subvariety of $\CC^l_{\zeta}$ for some $l$. 
Let $a=(a_1, \ldots, a_{n-1}) \in \ZZ^{n-1}$ 
be the coordinate of the vertex $\gamma$ of 
$\Delta$ and define a 
(non-trivial) rank one local system 
$\widetilde{\LL}$ on 
$T_0$ by 
\begin{equation} 
\widetilde{\LL} = \CC_{T_0} 
x_1^{c_1-c_na_1-1} \cdots x_{n-1}^{c_{n-1}
-c_na_{n-1}-1}. 
\end{equation}
Then on a neighborhood of 
the point $p$ in $U_{\gamma} \subset \CC^l_{\zeta}$, 
$Ri_* \M$ is isomorphic to $Ri_* \widetilde{\LL}$. 
Next, as in the proof of \cite[Lemma 4.2]{E-T-2} 
we construct a real-valued function 
$\varphi$ on $\CC^l_{\zeta}$ whose level 
sets $\Omega_t= \{ \zeta \in \CC^l \ | \ 
\varphi ( \zeta )<t \}$ ($t \in \RR_{>0}$) 
satisfy the conditions $\cap_{t>0} \Omega_t
= \{ p \} =X_{\gamma}$ and $(\cup_{t>0} \Omega_t) 
\cap T_0 
= T_0$ and use it to show the isomorphism 
\begin{equation} 
0 \simeq R \Gamma (T_0; \widetilde{\LL}) 
\simto (Ri_* \widetilde{\LL})_p  
\end{equation}
by the twisted Morse theory. 
We thus obtain the isomorphism 
$(Ri_* \M )_p \simeq 0 $. 
When the point $p \in 
X \setminus i(T_0)$ lies in a 
$T_0$-orbit $X_{\gamma}$ such that 
$\dim X_{\gamma} = \dim \gamma >0$, by 
taking a normal slice of $X_{\gamma}$ in $X$ 
we can reduce the problem to the case 
where $\dim X_{\gamma} =0$. We thus obtain an 
isomorphism 
$i_! \M \simeq Ri_* \M$ in $\Dbc (X)$. 
Applying the functor $R \Gamma_c(X; \cdot ) = 
R \Gamma (X; \cdot )$ to it we obtain the 
desired isomorphisms 
\begin{equation} 
H^j_c(T_0; \M ) \simeq H^j(T_0; \M ) 
\simeq H^j(W ; \LL ) 
\end{equation} 
for $j \in \ZZ$. Now recall that $T_0$ is 
an affine variety and $\M \in \Dbc (T_0)$ 
is a perverse sheaf on it (up to some 
shift). Then by Artin's vanishing theorem 
for perverse sheaves over affine varieties 
(see \cite[Corollaries 5.2.18 and 5.2.19]{Dimca} 
etc.) we have 
\begin{equation} 
H^j_c(T_0; \M ) \simeq 0 \quad \text{for} \ 
j< \dim T_0 =n-1
\end{equation} 
and 
\begin{equation} 
H^j(T_0; \M ) \simeq 0 \quad \text{for} \ 
j> \dim T_0 =n-1, 
\end{equation} 
from which the last assertion 
immediately follows. 
This completes the proof. 
\end{proof} 

By Theorem \ref{BKK} we obtain the following 
corollary of Theorem \ref{VTM}. 

\begin{corollary}\label{IMPC} 
In the situation of Theorem \ref{VTM}, 
let $p_1, \ldots, p_r \in P^{-1}(0)$ 
be the (isolated) singular points of 
$P^{-1}(0) \subset T_0$ and for 
$1 \leq i \leq r$ let $\mu_i>0$ be the 
Milnor number of $P^{-1}(0)$ at $p_i$. 
Then we have 
\begin{equation} 
\dim H^{n-1}(W ; \LL ) =
 \Vol_{\ZZ}( \Delta ) - 
\sum_{i=1}^r \mu_i. 
\end{equation} 
\end{corollary}
\begin{proof}
By multiplying a monomial $x^a$ ($a \in \ZZ^{n-1}$) 
to $P(x)$ we may assume that the Newton polytope 
$\Delta$ of $P$ contains the origin $0 \in \RR^{n-1}$. 
Then by Sard's theorem the generic fiber 
$P^{-1}(t) \subset T_0$ ($t \not= 0$) of the map 
$P:T_0 \longrightarrow \CC$ is a non-degenerate 
hypersurface of $T_0$ in the sense of 
Definition \ref{non-deg}. Hence it follows from 
Theorem \ref{BKK} that its Euler characteristic 
$\chi (P^{-1}(t))$ is equal to 
$(-1)^{n-1-1} \Vol_{\ZZ}( \Delta ) = 
(-1)^{n-2} \Vol_{\ZZ}( \Delta )$. It is also 
well-known that we have 
\begin{equation} 
\chi (P^{-1}(0))= 
\chi (P^{-1}(t)) -(-1)^{n-2} 
\sum_{i=1}^r \mu_i. 
\end{equation} 
For the open set $W=T_0 \setminus P^{-1}(0)$ 
of $T_0$, by $\chi (T_0)=0$ we thus obtain the 
equality 
\begin{equation} 
(-1)^{n-1} \chi (W)= 
\Vol_{\ZZ}( \Delta ) - 
\sum_{i=1}^r \mu_i. 
\end{equation} 
Moreover by applying the Mayer-Vietoris argument 
to the rank one local system $\LL$ we have 
$\chi (W)= \sum_{j \in \ZZ} (-1)^j 
\dim H^{j}(W ; \LL )$. Then the assertion 
follows immediately from 
Theorem \ref{VTM}. 
\end{proof}

We can generalize Theorem \ref{VTM} 
to the case where the hypersurface 
$S= \overline{i (P^{-1}(0))} 
\subset X$ has 
(stratified) isolated 
singular points $p$ also in 
$T_0$-orbits $X_{\gamma} 
\subset X \setminus i(T_0)$ as follows. 
For such a point 
$p \in S \cap X_{\gamma}$ of $S$ 
let us show that we have the 
vanishing $(Ri_* \M )_p 
\simeq 0$ in general. 
First consider the case where 
the codimension of $X_{\gamma}$ in $X$ is one. 
The question being local, 
it suffices to consider the case where 
$X= \CC^{n-1}_y \supset 
X_{\gamma}= \{ y_{n-1}=0 \}$, 
$S= \{ f(y)=0 \} \ni p=0$, 
$T_0= \CC^{n-1} \setminus \{ y_{n-1}=0 \}$, 
$i: \CC^{n-1} \setminus \{ y_{n-1}=0 \} \hookrightarrow 
\CC^{n-1}$ and 
\begin{equation} 
\LL = \CC_{\CC^{n-1} \setminus 
\{ f(y) \cdot y_{n-1}=0 \} } 
f(y)^{\alpha} y_{n-1}^{\beta} 
\end{equation}
for $\alpha = -c_n$ and some 
$\beta \in \CC$ 
(by the notation in the 
proof of Theorem \ref{VTM} we have 
$\beta = m ( \gamma )$). Here $f(y)$ is a 
polynomial on 
$\CC^{n-1}$ such that $S=f^{-1}(0)$ 
has a (stratified) isolated singular point 
at $p=0 \in  S \cap X_{\gamma}$. Moreover for 
the inclusion map 
$\iota : \CC^{n-1} 
\setminus \{ f(y) \cdot y_{n-1}=0 \}
\hookrightarrow \CC^{n-1} 
\setminus \{ y_{n-1}=0 \}$ 
we have $\M \simeq R \iota_* \LL$. 
By the nonresonance 
of $c \in \CC^n$ we have $\beta  = m ( \gamma ) 
\notin \ZZ$ and there exists an isomorphism 
\begin{equation}\label{EQ=1} 
i_! ( \CC_{\CC^{n-1} 
\setminus \{ y_{n-1}=0 \} }
 y_{n-1}^{\beta} ) 
\simto 
Ri_* ( \CC_{\CC^{n-1} 
\setminus \{ y_{n-1}=0 \} }
 y_{n-1}^{\beta} ) . 
\end{equation}
Set $\N = i_! ( \CC_{\CC^{n-1} 
\setminus \{ y_{n-1}=0 \} }
 y_{n-1}^{\beta} )$. Then $\N$ is a perverse 
sheaf (up to some shift) 
on $X= \CC^{n-1}$ and satisfies 
the condition 
$\psi_{f}( \N )_p 
\simeq \phi_{f}( \N )_p$ (use \eqref{EQ=1}), where 
\begin{equation} 
\psi_f, \phi_f : \Dbc(X) 
\longrightarrow \Dbc ( \{ f=0 \} ) 
\end{equation}
are the nearby and vanishing 
cycle functors associated to $f$ respectively 
(see \cite{Dimca} etc.). 
By the $t$-exactness 
of the functor $\phi_{f}$ 
the constructible 
sheaf $\phi_{f}( \N )$ 
on $S=f^{-1}(0)$ 
is perverse (up to some 
shift). Moreover by our assumption its 
support is contained in the 
point $\{ p \} = \{ 0 \} 
\subset X = \CC^{n-1}$. 
This implies that 
we have the concentration 
\begin{equation} 
H^j \psi_{f}( \N )_p 
\simeq H^j \phi_{f}( \N )_p 
\simeq 0 \qquad (j \not= n-2). 
\end{equation}
Namely for the Milnor fiber $F_p$ of $f$ at 
$p=0 \in \CC^{n-1}$ we have 
\begin{equation} 
H^j(F_p; \N ) \simeq 
H^j \psi_{f}( \N )_p 
\simeq 0 \qquad (j \not= n-2). 
\end{equation}
Let $B(p; \varepsilon ) \subset \CC^{n-1}$ 
be a small open ball in $\CC^{n-1}$ centered 
at $p=0$ and for $0< \eta \ll \varepsilon$ set 
\begin{equation} 
G= \{ y \in \overline{B(p; \varepsilon )} \ | \ 
0<| f( y )| < \eta \}. 
\end{equation}
Then, in order to show the vanishing 
$(Ri_* \M )_p \simeq 0$ it suffices to 
prove the one $R \Gamma (G; Ri_* \M ) \simeq 0$ 
for the constructible sheaf 
\begin{equation} 
(Ri_* \M )|_G \simeq 
( \N |_G) \otimes_{\CC_G} (f|_G)^{-1} \LL^{\prime}
\end{equation}
on $G$, where $\LL^{\prime}$ is the rank one 
local system on the punctured disk 
$D^*_{\eta}= 
\{ t \in \CC \ | \ 0 <|t|< \eta \} \subset \CC$ 
generated by the function $t^{\alpha}$. 
By the projection formula we have 
\begin{equation} 
R \Gamma (G; Ri_* \M ) \simeq 
R \Gamma ( D^*_{\eta}; R (f|_G)_*( \N |_G) 
\otimes_{\CC_{D^*_{\eta}}} \LL^{\prime}). 
\end{equation}
Note that $H^j R (f|_G)_*( \N |_G) \simeq 0$ 
($j \not= n-2$) and 
$H^{n-2} R (f|_G)_*( \N |_G)$ is a local 
system on $D^*_{\eta}$ whose stalks are 
isomorphic to $H^{n-2}(F_p; \N )
\simeq  H^{n-2} \psi_{f}( \N )_p$. 
Hence, in order to show the vanishing 
$(Ri_* \M )_p \simeq 0$ it suffices to prove 
that the monodromy operator 
$\Phi : H^{n-2} \psi_{f}( \N )_p \simto 
H^{n-2} \psi_{f}( \N )_p$ 
does not have the 
eigenvalue $\exp (- 2 \pi i \alpha )$. 
For this purpose, we shall use the results 
in \cite[Section 5]{M-T-2}. 
Let $\Gamma_+(f) 
\subset \RR_+^{n-1}$ be the convex hull of 
$\cup_{v \in \supp (f)} (v+ \RR^{n-1}_+)$ in $\RR^{n-1}_+$. 
We call it the Newton polyhedron of 
$f$ at the origin $p=0 \in 
\CC^{n-1}$. 

\begin{definition} (see \cite{khov0}, \cite{Oka} etc.) 
We say that $f$ is Newton non-degenerate 
at the origin $p=0 \in 
\CC^{n-1}$ if for any compact face $\gamma \prec 
\Gamma_+(f)$ of $\Gamma_+(f)$ 
the hypersurface 
$\{ y \in (\CC^*)^{n-1} \ | \ f^{\gamma}
(y)=0 \}$ of $(\CC^*)^{n-1}$ is 
smooth and reduced. 
\end{definition}
For each subset 
$I \subset \{ 1,2, \ldots, n-1 \}$ 
we set 
\begin{equation} 
\RR_+^I= \{ v=(v_1, \ldots, v_{n-1}) 
\in \RR_+^{n-1} 
 \ | \ v_i=0 \ 
\text{for any} \ i \notin I \} \simeq \RR_+^{\sharp I}. 
\end{equation}
Let $\gamma_1^I, 
\ldots, \gamma_{n(I)}^I \prec 
\Gamma_+(f) \cap \RR_+^I$ 
be the compact facets of 
$\Gamma_+(f) \cap \RR_+^I$. 
For $1 \leq i \leq n(I)$ 
denote by $d_i^I \in \ZZ_{>0}$ the lattice 
distance of $\gamma_i^I$ from the origin 
$0 \in \RR^{I}_+$ 
and let $u_i^I=(u_{i,1}^I, \ldots, 
u_{i,n-1}^I) \in 
\RR_+^I \cap \ZZ^{n-1}$ be the 
unique (non-zero) primitive 
vector which takes its 
minimum exactly on 
$\gamma_i^I$. For simplicity 
we set $\delta_i^I:=u_{i,n-1}^I$. 
Finally we define 
a finite subset 
$E_p \subset \CC$ of $\CC$ by 
\begin{equation} 
E_p= \bigcup_{I: I \ni n-1} 
\bigcup_{i=1}^{n(I)} 
\{ \lambda \in \CC \ | 
\ \lambda^{d_i^I} = 
\exp (2 \pi \sqrt{-1} 
\beta \cdot \delta_i^I) \}. 
\end{equation}
Then the following result 
is a special case of \cite[Theorem 5.5]{M-T-2}. 

\begin{proposition}
In the above situation, assume moreover that 
$f$ is Newton non-degenerate 
at the origin $p=0 \in 
\CC^{n-1}$. Then the set 
of the eigenvalues of 
the monodromy operator 
$\Phi : H^{n-2} 
\psi_{f}( \N )_p \simto 
H^{n-2} \psi_{f}( \N )_p$ 
is contained in $E_p$. 
\end{proposition} 

\begin{corollary}
Assume that $\dim \Delta =n-1$, 
$c \in \CC^n$ is nonresonant, 
$\exp (- 2 \pi \sqrt{-1}  \alpha ) 
= \exp ( 2 \pi \sqrt{-1}  c_n ) 
\notin E_p$ and 
$f$ is Newton non-degenerate 
at the origin $p=0 \in \CC^{n-1}$.  Then we have 
$(Ri_* \M )_p \simeq 0$. 
\end{corollary}
In fact, by \cite[Theorem 5.5]{M-T-2} 
we can generalize this corollary to the case 
where the codimension of the $T_0$-orbit 
$X_{\gamma}$ in $X_{\gamma} 
\subset X \setminus i(T_0)$ containing 
the (stratified) isolated 
singular point $p$ of $S$ is larger than one. 
We leave the precise formulation 
to the reader and omit the details here. 
In this way, our Theorem \ref{VTM} can be generalized 
to the case where $S$ has (stratified) isolated 
singular points $p$ also in $T_0$-orbits $X_{\gamma} 
\subset X \setminus i(T_0)$. 
In particular we have the following result. 
For a face $\gamma$ of $\Delta$ let 
$L_{\gamma} \simeq \RR^{\dim \gamma}$ 
be the linear subspace of $\RR^{n-1}$ 
parallel to the affice span of $\gamma$ 
in $\RR^{n-1}$ and consider the $\gamma$-part 
$P^{\gamma}$ of $P$ as a function on 
$T_{\gamma}= \Spec ( \CC [ L_{\gamma} \cap 
\ZZ^{n-1} ] ) \simeq ( \CC^*)^{\dim \gamma}$. 

\begin{theorem}\label{SVTM} 
Assume that $\dim \Delta =n-1$ and 
for any face $\gamma$ of $\Delta$ 
the hypersurface $(P^{\gamma})^{-1}(0) 
\subset T_{\gamma}$ of $T_{\gamma}$ 
has only isolated singular points. 
Then for generic parameter vectors 
$c \in \CC^n$ we have the 
concentration 
\begin{equation} 
H^j(W ; \LL ) \simeq 
0 \qquad (j \not= n-1). 
\end{equation} 
\end{theorem}

\medskip \par 

From now, let us generalize Theorem \ref{VTM} 
to the following more general situation. 
For $0<k<n$ let $B_i=\{ b_i(1), b_i(2), \ldots , 
b_i(N_i)\} \subset \ZZ^{n-k}$ 
($1 \leq i \leq k$) be $k$ finite subsets 
of the lattice $\ZZ^{n-k}$ and set 
$N=N_1 + N_2 + \cdots +N_k$. 
For $1 \leq i \leq k$ and 
$(z_{i1}, \ldots, z_{iN_i}) \in \CC^{N_i}$ 
we define a Laurent polynomial $P_i(x)$ 
on $T_0=( \CC^*)^{n-k}$ by 
$P_i(x)=\sum_{j=1}^{N_i} z_{ij} x^{b_i(j)}$ 
($x=(x_1, \ldots, x_{n-k}) \in T_0=( \CC^*)^{n-k}$). 
Let us set $W=T_0 \setminus \cup_{i=1}^k P_i^{-1}(0)$. 
Then for 
$c=(c_1, \ldots, c_{n-k}, \tl{c_1}, 
\ldots, \tl{c_k}) \in \CC^n$ 
the possibly multi-valued function 
\begin{equation} 
P_1(x)^{-\tl{c_1}} \cdots P_k(x)^{-\tl{c_k}} 
x_1^{c_1-1} \cdots x_{n-k}^{c_{n-k}-1}
\end{equation} 
on $W$ generates the local system 
\begin{equation} 
\LL = \CC_{W} 
P_1(x)^{-\tl{c_1}} \cdots P_k(x)^{-\tl{c_k}} 
x_1^{c_1-1} \cdots x_{n-k}^{c_{n-k}-1}. 
\end{equation} 
Let $e_i=(0,0, \ldots, 0,1,0, \ldots, 0) \in \ZZ^k$ 
($1 \leq i \leq k$) be the standard basis of $\ZZ^k$ 
and set $a_i(j)=(b_i(j), e_i) \in \ZZ^{n-k} 
\times \ZZ^k=\ZZ^n$ 
($1 \leq i \leq k$, $1 \leq j \leq N_i$) 
and 
\begin{equation} 
A=\{ a_1(1), \ldots , a_1(N_1), 
\ldots \ldots , 
a_k(1), \ldots , a_k(N_k)
\} \subset \ZZ^{n}. 
\end{equation} 
For $1 \leq i \leq k$ let $\Delta_i \subset 
\RR^{n-k}$ be the convex 
hull of $B_i$ in $\RR^{n-k}$. 
Denote by $\Delta \subset \RR^{n-k}$ 
their Minkowski sum $\Delta_1 + \cdots +
\Delta_k$. Assume that $\dim \Delta =n-k$. 
Then by using the $n$-dimensional 
closed convex polyhedral cone 
$K= \RR_+ A \subset \RR^n$ 
generated by $A$ in $\RR^n$ 
we can define the nonresonance of 
the parameter $c \in \CC^n$ as in 
Definition \ref{NRC}. 
For a face $\gamma \prec \Delta$ 
of $\Delta$ let $\gamma_i \prec \Delta_i$ 
be the faces of $\Delta_i$ ($1 \leq i \leq k$) 
canonically associated to $\gamma$ such that 
$\gamma = \gamma_1 + \cdots +
\gamma_k$. 

\begin{definition}\label{MND} 
We say that the $k$-tuple of the 
Laurent polynomials 
$(P_1, \ldots, P_k)$ is 
``weakly" (resp. ``strongly") non-degenerate if 
for any face $\gamma$ of 
$\Delta$ such that 
$\dim \gamma < \dim \Delta = n-k$ 
(resp. $\dim \gamma \leq \dim \Delta = n-k$) 
and non-empty subset 
$J \subset \{ 1,2, \ldots, k \}$ 
the subvariety 
\begin{equation} 
\{ x \in T_0=(\CC^*)^{n-k} \ | \ 
P_i^{\gamma_i}(x)=0 \ \ 
(i \in J) \} \subset T_0 
\end{equation}
is a non-degenerate complete intersection. 
\end{definition}

\begin{remark}
Denote the convex hull of 
$\cup_{i=1}^k (\Delta_i \times \{ e_i \}) \subset 
\RR^{n-k} \times \RR^{k} = 
\RR^{n}$ in $\RR^{n}$ by $\Delta_1 * \cdots * \Delta_k$. 
Then $\Delta_1 * \cdots * \Delta_k$ is naturally identified 
with the Newton polytope of the Laurent polynomial 
$R(x,t)=\sum_{i=1}^k P_i(x)t_i $ on 
$\widetilde{T_0}:= T_0 \times (\CC^*)_t^k \simeq 
(\CC^*)_{x,t}^{n}$. In \cite{G-K-Z-2} the authors considered 
the condition that for any face $\gamma$ of 
$\Delta_1 * \cdots * \Delta_k$ the hypersurface 
$\{ (x,t) \in \widetilde{T_0} \ | \ 
R^{\gamma}(x,t)=0 \} \subset \widetilde{T_0}$ of 
$\widetilde{T_0}$ is smooth and reduced.  It is easy to 
see that our strong non-degeneracy of the 
$k$-tuple $(P_1, \ldots, P_k)$ 
in Definition \ref{MND} 
is equivalent to their condition. 
\end{remark} 

Let $\iota : 
W=T_0 \setminus \cup_{i=1}^k P_i^{-1}(0) 
 \longrightarrow T_0$ be the inclusion 
map and set $\M = R \iota_* \LL 
\in \Dbc (T_0)$. 

\begin{theorem}\label{MVTM} 
Assume that $\dim \Delta = n-k$, 
the parameter vector 
$c \in \CC^n$ is nonresonant and 
$(P_1, \ldots, P_k)$ is 
weakly non-degenerate. Then there exists 
an isomorphism 
\begin{equation} 
H^j_c(T_0; \M ) \simeq H^j(T_0; \M ) 
\simeq H^j(W ; \LL ) 
\end{equation} 
for any $j \in \ZZ$. Moreover we have the 
concentration 
\begin{equation} 
H^j(W ; \LL ) \simeq 
0 \qquad (j \not= n-k). 
\end{equation} 
\end{theorem}
\begin{proof} 
The proof is similar to that of Theorem \ref{VTM}. 
Let $\Sigma_0$ be the dual fan of $\Delta$ in $\RR^{n-k}$ 
and $X$ the (possibly singular) toric 
variety associated to it. 
For a face $\gamma$ of $\Delta$ 
we denote by $X_{\gamma} \simeq 
(\CC^*)^{\dim \gamma}$ the $T_0$-orbit 
associated to $\gamma$. Let $i : 
X_{\Delta} \simeq T_0 \hookrightarrow X$ be 
the inclusion map. Then by the weak 
non-degeneracy of the $k$-tuple 
$(P_1, \ldots, P_k)$, for 
any $T_0$-orbits $X_{\gamma}$ 
in $X \setminus X_{\Delta}$ and the closure 
$S= \overline{i (\cup_{i=1}^k P_i^{-1}(0))} 
\subset X$ 
of the hypersurface $i(\cup_{i=1}^k 
P_i^{-1}(0)) \subset i (T_0)$ 
in $X$ their intersection 
$S \cap X_{\gamma} 
\subset X_{\gamma}$ is a normal 
crossing divisor in $X_{\gamma}$. In fact 
$S$ itself is normal crossing on a neighborhood 
of such $X_{\gamma}$ and any irreducible 
component of it intersects $X_{\gamma}$ 
transversally. 
Moreover by the nonresonance 
of $c \in \CC^n$, for any 
$\gamma \prec \Delta$ such that $\dim 
\gamma =n-k-1$ the monodromy of the local 
system $\LL$ around the codimension-one 
$T_0$-orbit $X_{\gamma} \subset X$ in $X$ 
is non-trivial. Indeed, let $\gamma \prec \Delta$ 
be such a facet of $\Delta$ and 
$\gamma_i \prec \Delta_i$ 
the faces of $\Delta_i$ ($1 \leq i \leq k$) 
associated to $\gamma$ such that 
$\gamma = \gamma_1 + \cdots +
\gamma_k$. We denote the convex hull of 
$\cup_{i=1}^k (\Delta_i \times \{ e_i \})$ 
(resp. $\cup_{i=1}^k (\gamma_i \times \{ e_i \})$) 
$\subset \RR^{n-k} \times \RR^{k} = \RR^{n}$ 
in $\RR^{n}$ by $\Delta_1 * \cdots * \Delta_k$ 
(resp. $\gamma_1 * \cdots * \gamma_k$). 
Then $\Delta_1 * \cdots * \Delta_k$ is the 
join of $\Delta_1, \ldots, \Delta_k$ and 
$\gamma_1 * \cdots * \gamma_k$ is its 
facet. We denote by $\Gamma$ the facet of the cone 
$K= \RR_+ A$ generated by 
$\gamma_1 * \cdots * \gamma_k \subset K$. 
Let $\nu \in \ZZ^{n-k} 
\setminus \{ 0 \}$ be the primitive 
inner conormal vector of the facet 
$\gamma$ of $\Delta \subset 
\RR^{n-k}$ and for $1 \leq i \leq k$ set 
\begin{equation} 
m_i= \min_{v \in \Delta_i} 
\langle \nu, v \rangle = 
\min_{v \in \gamma_i} 
\langle \nu, v \rangle \in \ZZ. 
\end{equation} 
Then the primitive 
inner conormal vector 
$\widetilde{\nu} \in \ZZ^{n} 
\setminus \{ 0 \}$ of the facet 
$\Gamma$ of $K \subset \RR^{n}$ 
is explicitly given by the formula 
\begin{equation} 
\widetilde{\nu} = 
\left(  \begin{array}{c}
      \nu \\
      -m_1 \\
      \vdots \\
      -m_k 
    \end{array}  \right) \in \ZZ^{n} 
\setminus \{ 0 \}. 
\end{equation}
and the condition 
$c=(c_1, \ldots, c_{n-k}, \tl{c_1}, 
\ldots, \tl{c_k}) \notin 
\{ \ZZ^n+ \Lin (\Gamma ) \}$ is 
equivalent to the one 
\begin{equation} 
m( \gamma ):= 
\biggl\langle \nu , \quad 
\left(  \begin{array}{c}
      c_1-1 \\
      \vdots \\
      c_{n-k}-1 
    \end{array}  \right)  \biggr\rangle 
- \sum_{i=1}^k m_i \cdot \tl{c_i} 
\quad  \notin \ZZ. 
\end{equation} 
Moreover we can easily see that the 
order of the (multi-valued) function 
\begin{equation} 
P_1(x)^{-\tl{c_1}} \cdots P_k(x)^{-\tl{c_k}} 
x_1^{c_1-1} \cdots x_{n-k}^{c_{n-k}-1}
\end{equation} 
along the codimension-one 
$T_0$-orbit $X_{\gamma} \subset X$ in $X$ 
is equal to $m( \gamma ) \notin \ZZ$. 
Finally, by constructing suitable distance functions 
as in the proof of 
\cite[Lemma 4.2]{E-T-2}, we can show that 
\begin{equation}
(Ri_* \M )_p \simeq 0 \qquad \text{for any} \ p \in 
X \setminus T_0. 
\end{equation}
Namely there exists an isomophism 
$i_! \M \simeq Ri_* \M$ in $\Dbc (X)$. 
Applying the functor $R \Gamma_c(X; \cdot ) = 
R \Gamma (X; \cdot )$ to it we obtain the 
desired isomorphisms 
\begin{equation} 
H^j_c(T_0; \M ) \simeq H^j(T_0; \M ) 
\simeq H^j(W ; \LL ) 
\end{equation} 
for $j \in \ZZ$. Then 
the remaining assertion can be proved 
as in the proof of Theorem \ref{VTM}. 
This completes the proof. 
\end{proof} 

In the situation of Theorem \ref{MVTM}, for any 
$1 \leq i \leq k$ the 
hypersurface $P_i^{-1}(0) \subset T_0$ 
has only isolated singular points. 
Assume moreover that the hypersurface 
$\cup_{i=1}^k P_i^{-1}(0) \subset T_0$ 
is normal crossing outside them. 
Then as in Corollary \ref{IMPC}, 
by Theorem \ref{BKK} we can also express 
the dimension of $H^{n-k}(W ; \LL )$ 
in terms of some mixed volumes of 
the polytopes $\Delta_1, \ldots, \Delta_k$ 
and the Milnor numbers of the 
isolated singular points. Since the 
statement of this result is involved, 
we leave its precise formulation to the reader. 

\medskip \par 

As in the case where $k=1$ we have the following 
result. For a face $\gamma$ of $\Delta$ let 
$L_{\gamma} \simeq \RR^{\dim \gamma}$ 
be the linear subspace of $\RR^{n-k}$ 
parallel to the affice span of $\gamma$ 
in $\RR^{n-k}$ and for $1 \leq i \leq k$ 
consider the $\gamma_i$-part 
$P_i^{\gamma_i}$ of $P_i$ as a function on 
$T_{\gamma}= \Spec ( \CC [ L_{\gamma} \cap 
\ZZ^{n-k} ] ) \simeq ( \CC^*)^{\dim \gamma}$. 

\begin{theorem}\label{SMVTM} 
Assume that $\dim \Delta = n-k$ and 
for any $1 \leq i \leq k$ the hypersurface 
$P_i^{-1}(0) \subset T_0$ of $T_0$ has only 
isolated singular points. Assume moreover that 
for any face $\gamma$ of 
$\Delta$ such that $\dim \gamma < 
\dim \Delta = n-k$ 
and non-empty subset 
$J \subset \{ 1,2, \ldots, k \}$ 
the $k$-tuple of the Laurent polynomials 
$(P_1, \ldots, P_k)$ satisfies the following 
condition: 
\medskip \par 
If $J=\{ i \}$ for some $1 \leq i \leq k$ 
and $\dim \gamma_i= \dim \gamma =
\dim \Delta -1= n-k-1$ 
the \par 
hypersurface 
$(P_i^{\gamma_i})^{-1}(0) 
\subset T_{\gamma}$ of $T_{\gamma}$ 
has only isolated singular points. 
Otherwise, 
\par 
the subvariety 
\begin{equation} 
\{ x \in T_0=(\CC^*)^{n-k} \ | \ 
P_i^{\gamma_i}(x)=0 \ \ 
(i \in J) \} \subset T_0 
\end{equation}
\par 
of $T_0$ is a non-degenerate 
complete intersection. 
\medskip \par 
\noindent Then for generic parameter vectors 
$c \in \CC^n$ we have the 
concentration 
\begin{equation} 
H^j(W ; \LL ) \simeq 
0 \qquad (j \not= n-k). 
\end{equation} 
\end{theorem}
\begin{proof} 
Let $\Sigma_0$ be the dual fan of $\Delta$ in $\RR^{n-k}$ 
and $X$ the (possibly singular) toric 
variety associated to it. Then our assumptions imply 
that for any $1 \leq i \leq k$ the hypersurface 
$S_i= \overline{i ( P_i^{-1}(0))} 
\subset X$ has only stratified 
isolated singular points in $X$ and we can prove 
the assertion following the proofs of Theorems 
\ref{SVTM} and \ref{MVTM}. 
\end{proof} 

For a face $\gamma$ of 
$\Delta$ and $1 \leq i \leq k$ such that $\dim \gamma_i < 
\dim \gamma \leq n-k-1$ 
the hypersurface 
$(P_i^{\gamma_i})^{-1}(0) 
\subset T_{\gamma}$ of $T_{\gamma}$ is smooth or 
has non-isolated singularities. In the latter case, 
we cannot prove the concentration in 
Theorem \ref{SMVTM} by our methods. 
This is the reason why we do not allow 
such cases in our assumptions of Theorem \ref{SMVTM}. 
However, in the very special case where the Newton polytopes 
$\Delta_1, \ldots, \Delta_k$ are similar each other, 
we do not have this problem and obtain  
the following simpler result. 

\begin{theorem}\label{SSMVTM} 
Assume that $\dim \Delta = n-k$, the Newton polytopes 
$\Delta_1, \ldots, \Delta_k$ are similar each other and 
for any face $\gamma$ of $\Delta$ and $1 \leq i \leq k$ 
the hypersurface $(P_i^{\gamma_i})^{-1}(0) 
\subset T_{\gamma}$ of $T_{\gamma}$ 
has only isolated singular points. Assume moreover that 
for any face $\gamma$ of 
$\Delta$ such that $\dim \gamma < 
\dim \Delta = n-k$ 
and any subset 
$J \subset \{ 1,2, \ldots, k \}$ such that 
$\sharp J \geq 2$ the subvariety 
\begin{equation} 
\{ x \in T_0=(\CC^*)^{n-k} \ | \ 
P_i^{\gamma_i}(x)=0 \ \ 
(i \in J) \} \subset T_0 
\end{equation}
of $T_0$ is a non-degenerate 
complete intersection. 
Then for generic parameter vectors 
$c \in \CC^n$ we have the 
concentration 
\begin{equation} 
H^j(W ; \LL ) \simeq 
0 \qquad (j \not= n-k). 
\end{equation} 
\end{theorem}

\section{Some results on the twisted 
Morse theory}\label{sec:4}

In this section, we prepare some 
auxiliary results on the twisted 
Morse theory which will be used in 
Section \ref{sec:5}. 
The following proposition 
is a refinement of the 
results in \cite[page 10]{Esterov}. 
See also \cite[Proposition 7.1]{E-T-2}. 

\begin{proposition}\label{MIS} 
Let $T$ be an algebraic torus $( \CC^*)^n_x$ 
and $T= \sqcup_{\alpha} Z_{\alpha}$ its 
algebraic stratification. In particular 
we assume that each stratum 
$Z_{\alpha}$ in it is smooth. 
Let $h(x)$ be a Laurent 
polynomial on $T=( \CC^*)^n_x$ 
such that the hypersurface $\{ h=0 \} \subset 
T$ intersects $Z_{\alpha}$ 
transversally for 
any $\alpha$. For 
$a \in \CC^n$ consider the 
(possibly multi-valued) function $g_a(x):= 
h(x) x^{-a}$ on $T$. 
Then there exists a non-empty 
Zariski open subset $\Omega \subset \CC^n$ of 
$\CC^n$ such that the restriction 
$g_a|_{Z_{\alpha}}: 
Z_{\alpha} \longrightarrow 
\CC$ of $g_a$ to $Z_{\alpha}$ 
has only isolated 
non-degenerate (i.e. 
Morse type) critical points 
for any $a \in \Omega 
\subset \CC^n$ and $\alpha$. 
\end{proposition}

\begin{proof}
We may assume that each stratum $Z_{\alpha}$ 
is connected. We fix a stratum $Z_{\alpha}$ 
and set $k= \dim Z_{\alpha}$. For a 
subset $I \subset \{ 1,2, \ldots, n \}$ such that 
$|I|=k= \dim Z_{\alpha}$ 
denote by $\pi_I : T=( \CC^*)^n_x 
\longrightarrow ( \CC^*)^k$ the projection 
associated to $I$. We also denote by 
$Z_{\alpha, I} \subset Z_{\alpha}$ the 
maximal Zariski open subset of $Z_{\alpha}$ 
such that the restriction of $\pi_I$ to 
it is locally biholomorphic. 
By the implicit function theorem, the 
variety $Z_{\alpha}$ is covered by such 
open subsets $Z_{\alpha, I}$. For simplicity, 
let us consider the 
case where $I= \{ 1,2, \ldots, 
k \}  
\subset \{ 1,2, \ldots, n \}$. Then 
we may regard $g_a|_{Z_{\alpha}}$ 
locally as a 
function $g_{a, \alpha, I}
(x_1, \ldots, x_k)$ on 
the Zariski open subset $\pi_I (Z_{\alpha, I}) 
\subset ( \CC^*)^k$ of the form 
\begin{equation} 
g_{a, \alpha, I}(x_1, \ldots, x_k) = 
\frac{h_{a, \alpha, I}(x_1, \ldots, 
x_k)}{x_1^{a_1} \cdots x_k^{a_k}}. 
\end{equation} 
By our assumption, the hypersurface 
$\{ h_{a, \alpha, I}=0 \} \subset 
\pi_I (Z_{\alpha, I}) \subset ( \CC^*)^k$ 
is smooth. Then as in the proof of 
\cite[Proposition 7.1]{E-T-2} we can show 
that there exists a non-empty 
Zariski open subset 
$\Omega_{\alpha, I} \subset \CC^n$ 
such that the (possibly multi-valued) function 
$g_{a, \alpha, I}(x_1, \ldots, x_k)$ on 
$\pi_I (Z_{\alpha, I}) \subset ( \CC^*)^k$ 
has only isolated 
non-degenerate (i.e. Morse type) critical points 
for any $a \in \Omega_{\alpha, I} 
\subset \CC^n$. This completes the proof. 
\end{proof}

\begin{corollary}\label{NCR} 
In the situation of Proposition \ref{MIS}, 
assume moreover that for the Newton polytope 
$NP(h) \subset \RR^n$ of $h$ we have 
$\dim NP(h)=n$. Then there exists 
$a \in \Int NP(h)$ such that the restriction 
$g_a|_{Z_{\alpha}}: Z_{\alpha} \longrightarrow 
\CC$ of $g_a$ to $Z_{\alpha}$ has only isolated 
non-degenerate (i.e. Morse type) critical points 
for any $\alpha$. 
\end{corollary}

Now let $Q_1, \ldots, Q_l$ be 
Laurent polynomials on $T=( \CC^*)^{n}$ 
and for $1 \leq i \leq l$ denote by 
$\Delta_i \subset \RR^{n}$ the Newton 
polytope $NP(Q_i)$ of $Q_i$. 
Set $\Delta = \Delta_1 + \cdots + \Delta_l$.
Then by Corollary \ref{NCR} 
we obtain the following result 
which might be of independent interest.  

\begin{theorem}\label{NVTM} 
Let $\LL$ be a non-trivial local system 
of rank one on $T=( \CC^*)^{n}$. 
Assume that for any $1 \leq i \leq l$ 
we have $\dim \Delta_i =n$ and the subvariety 
\begin{equation} 
Z_i = \{ x \in T \ | \ 
Q_1(x)= \cdots = Q_i(x)=0 \} \subset T 
\end{equation} 
of $T$ is a non-degenerate complete 
intersection. 
Then for any $1 \leq i \leq l$ we have the 
concentration 
\begin{equation} 
H^j(Z_i ; \LL ) \simeq 
0 \qquad (j \not= n-i). 
\end{equation} 
Moreover we have 
\begin{equation}
\dim H^{n-i} (Z_i ; \LL ) =  
\dsum_{\begin{subarray}{c} 
m_1,\ldots,m_i \geq 1\\ 
m_1+\cdots +m_i=n 
\end{subarray}}\Vol_{\ZZ}(
\underbrace{\Delta_1,
\ldots,\Delta_1}_{\text{
$m_1$-times}},\ldots, 
\underbrace{\Delta_i,
\ldots,\Delta_i}_{\text{$m_i$-times}}). 
\end{equation}
\end{theorem}

\begin{proof} 
We prove the assertion 
by induction on $i$. For $i=0$ 
we have $Z_i=T$ and the assertion is 
obvious. Since $Z_i \subset T$ is affine, 
by Artin's vanishing theorem we have 
the concentration 
\begin{equation}\label{avt} 
H^j(Z_i ; \LL ) \simeq 
0 \qquad (j > n-i= \dim Z_i). 
\end{equation} 
On the other hand, by Corollary \ref{NCR} 
there exists  
$a_i \in \Int NP(Q_i) \subset \RR^n$ 
such that the real-valued function 
\begin{equation} 
g_i: Z_{i-1} \longrightarrow \RR, \qquad 
x \longmapsto 
 | Q_i(x) x^{-a_i} | 
\end{equation} 
has only isolated non-degenerate (Morse type) 
critical points. Note that 
the Morse index of $g_i$ 
at each critical point is 
$\dim Z_{i-1} = n-i+1$.
Let $\Sigma_0$ be the dual fan of the 
$n$-dimensional polytope $\Delta$ in 
$\RR^n$ and $\Sigma$ its smooth 
subdivision. We denote by $X_{\Sigma}$ 
the toric variety associated to $\Sigma$. 
Then $X_{\Sigma}$ is a smooth compactification 
of $T$ such that $D= X_{\Sigma} \setminus T$ 
is a normal crossing divisor in it. 
By our assumption, the closures  
$\overline{Z_{i-1}}, \overline{Z_{i}}
\subset X_{\Sigma}$ 
of $Z_{i-1}, Z_i$ 
in $X_{\Sigma}$ are smooth.
Moreover they intersect $D$ etc. 
transversally.  
Let $U$ be a sufficiently small tubular 
neighborhood of $\overline{Z_i} 
\cap D$ in $\overline{Z_{i-1}}$. Then by 
\cite[Section 3.5]{Zaharia} (see also 
\cite{L-S}), for any $t \in \RR_+$ 
there exist isomorphisms 
\begin{equation} 
H^j( \{ g_i<t \} ; \LL ) \simeq 
H^j( \{ g_i<t \} \setminus 
U ; \LL ) 
 \qquad (j \in \ZZ ). 
\end{equation} 
Moreover the level set 
$g_i^{-1}(t) \cap (Z_{i-1} \setminus 
U)$ of $g_i$ in $Z_{i-1} \setminus 
U$ is compact in $Z_{i-1}$ 
and intersects $\partial U$ transversally 
for any $t \in \RR_+$. 
Hence for $t \gg 0$ we 
have isomorphisms 
\begin{equation} 
H^j( \{ g_i<t \} ; \LL ) 
\simeq H^j(Z_{i-1} ; \LL ) 
 \qquad (j \in \ZZ ). 
\end{equation} 
Moreover for $0 < t \ll 1$ 
we have isomorphisms 
\begin{equation} 
H^j( \{ g_i<t \} ; \LL ) 
\simeq H^j(Z_{i} ; \LL ) 
 \qquad (j \in \ZZ ). 
\end{equation} 
When $t \in \RR$ decreases passing through 
one of the critical values of $g_i$, only 
the dimensions of 
$H^{n-i+1}( \{ g_i<t \} ; \LL )$ and  
$H^{n-i}( \{ g_i<t \} ; \LL )$ 
may change and 
the other cohomology groups 
$H^j( \{ g_i<t \} ; \LL )$  
$(j \not= n-i+1, n-i)$ 
remain the same. Then by our induction 
hypothesis for $i-1$ and \eqref{avt} we obtain 
the desired concentration 
\begin{equation} 
H^j(Z_i ; \LL ) \simeq 
0 \qquad (j \not= n-i). 
\end{equation} 
Moreover the last assertion follows 
from Theorem \ref{thm:2-14}. 
This completes the proof. 
\end{proof} 

From now on, assume also that the $l$-tuple 
$(Q_1, \ldots, Q_l)$ is 
strongly non-degenerate and 
$\dim \Delta_l=n$. Let $T= \sqcup_{\alpha} 
Z_{\alpha}$ be the algebraic stratification 
of $T$ associated to the hypersurface 
$S= \cup_{i=1}^{l-1} Q_i^{-1}(0) \subset 
T$ and set $M=T \setminus S$. Then 
by Corollary \ref{NCR}  
there exists $a \in \Int ( \Delta_l)$ 
such that the restriction 
of the (possibly multi-valued) function 
$Q_l(x)x^{-a}$ 
to $Z_{\alpha}$ has only isolated 
non-degenerate (i.e. Morse type) 
critical points 
for any $\alpha$. In particular, it has only 
stratified isolated singular points. 
We fix such $a \in \Int ( \Delta_l)$ 
and define a real-valued function 
$g: T \longrightarrow \RR_+$ by 
$g(x)=| Q_l(x)x^{-a} |$. 
For $t \in \RR_+$ we set also 
\begin{equation} 
M_t= \{ x \in M=T \setminus S \ | \ 
g(x)<t \} \subset M. 
\end{equation} 
Then we have the following result. 

\begin{lemma}\label{VL} 
Let $\LL$ be a local system on 
$M=T \setminus S$. 
Then for any $c >0$ 
there exists a sufficiently small 
$0 < \e \ll 1$ such that we have the 
concentration 
\begin{equation} 
H^j( M_{c+ \e}, M_{c- \e} ; \LL ) \simeq 
0 \qquad (j \not= n). 
\end{equation} 
\end{lemma}

\begin{proof} 
Let $\Sigma_0$ be the dual fan of the 
$n$-dimensional polytope $\Delta$ in 
$\RR^n$ and $\Sigma$ its smooth 
subdivision. We denote by $X_{\Sigma}$ 
the toric variety associated to $\Sigma$. 
Then $X_{\Sigma}$ is a smooth compactification 
of $T$ such that $D= X_{\Sigma} \setminus T$ 
is a normal crossing divisor in it. 
By the strong non-degeneracy 
of $(Q_1, \ldots, Q_l)$, the hypersurface 
$\overline{Q_l^{-1}(0)} \subset X_{\Sigma}$ 
intersects $D$ etc. transversally.
Let $U$ be a sufficiently small tubular 
neighborhood of $\overline{Q_l^{-1}(0)} 
\cap D$ in $X_{\Sigma}$ and for $t \in \RR_+$ 
set $M_t^{\prime}=M_t \setminus 
U$. Then by 
\cite[Section 3.5]{Zaharia},  
for any $t \in \RR_+$ there exist isomorphisms 
\begin{equation} 
H^j( M_t ; \LL ) \simeq 
H^j( M_t^{\prime} ; \LL )
 \qquad (j \in \ZZ ). 
\end{equation} 
Moreover the level set 
$g^{-1}(t) \cap (T \setminus 
U)$ of $g$ in $T \setminus 
U$ is compact in $T$ 
and intersects $\partial U$ transversally 
for any $t \in \RR_+$. 
For $c>0$ let $p_1, \ldots, p_r  
\in T \setminus g^{-1}(0)= 
T \setminus Q_l^{-1}(0)$ 
be the stratified isolated singular 
points of the function 
$h(x)=Q_l(x)x^{-a}$ in $T$ 
such that $g(p_i)=|h(p_i)|=c$.
Note that we have 
\begin{equation} 
g(x)=|h(x)|= \exp [ {\rm Re} 
\{ \log h(x) \} ].  
\end{equation}
Then there exist small open balls 
$B_i$ centered at $p_i$ in $T$ and 
$0 < \e \ll 1$ such 
that we have isomorphisms 
\begin{equation} 
H^j( M_{c+ \e}^{\prime}, 
M_{c- \e}^{\prime} ; \LL ) 
\simeq \bigoplus_{i=1}^r 
H^j( B_i \cap M_{c+ \e}, 
B_i \cap M_{c- \e} ; \LL )
 \qquad (j \in \ZZ ). 
\end{equation}
For $1 \leq i \leq r$ by taking a 
local branch $\log h$ of the 
logarithm of the function $h \not= 0$ on a 
neighborhood of $p_i \in T \setminus 
h^{-1}(0)$ we set 
$f_i= \log h - \log h(p_i)$. 
Then $f_i$ has also a stratified 
isolated singular point at $p_i$. 
Let $F_i \subset B_i$ be the Milnor fiber 
of $f_i$ at $p_i \in f_i^{-1}(0)$. 
Then for any $1 \leq i \leq r$ 
by shrinking $B_i$ 
if necessary we can 
easily prove the isomorphisms 
\begin{equation} 
H^j( B_i \cap M_{c+ \e}, 
B_i \cap M_{c- \e} ; \LL ) 
\simeq 
H^j( B_i \setminus S, 
F_i \setminus S ; \LL ) 
 \qquad (j \in \ZZ ). 
\end{equation}
Let $j : M=T \setminus S \hookrightarrow T$ 
be the inclusion. Since the Milnor fibers 
$F_i \subset B_i$ intersect 
each stratum $Z_{\alpha}$ 
transversally, we have also isomorphisms 
\begin{equation} 
H^j( B_i \setminus S, 
F_i \setminus S ; \LL ) 
\simeq H^{j-1} \phi_{f_i}
( Rj_* \LL )_{p_i} 
 \qquad (j \in \ZZ ), 
\end{equation}
where $\phi_{f_i}$ are 
Deligne's vanishing cycle 
functors. Hence by (the proof of) 
\cite[Proposition 6.1.1]{Dimca} 
the assertion follows from 
\begin{equation} 
{\rm supp} \  \phi_{f_i}
( Rj_* \LL ) \subset \{ p_i \} 
 \qquad (1 \leq i \leq r) 
\end{equation} 
and the fact that 
$Rj_* \LL$ and 
$\phi_{f_i}( Rj_* \LL )$ 
are perverse sheaves 
(up to some shifts). This completes 
the proof. 
\end{proof}

\section{A new vanishing theorem}\label{sec:5}

Now let $P_1, \ldots, P_k$ be 
Laurent polynomials on $T_0=( \CC^*)^{n-k}$ 
and for $1 \leq i \leq k$ denote by 
$\Delta_i \subset \RR^{n-k}$ the Newton 
polytope $NP(P_i)$ of $P_i$. 
Set $\Delta = \Delta_1 + \cdots + \Delta_k$. 
Let us set $W=T_0 \setminus 
\cup_{i=1}^k P_i^{-1}(0)$ 
and for $(c, \tl{c} ) 
=(c_1, \ldots, c_{n-k}, \tl{c_1}, 
\ldots, \tl{c_k}) \in \CC^n$ 
consider the local system 
\begin{equation} 
\LL = \CC_{W} 
P_1(x)^{\tl{c_1}} \cdots P_k(x)^{\tl{c_k}} 
x_1^{c_1} \cdots x_{n-k}^{c_{n-k}} 
\end{equation} 
on $W$. 

\begin{theorem}\label{NTM} 
Assume that the $k$-tuple of the 
Laurent polynomials $(P_1, \ldots, P_k)$ 
is strongly non-degenerate, 
$(c, \tl{c} ) =
(c_1, \ldots, c_{n-k}, \tl{c_1}, 
\ldots, \tl{c_k}) \notin \ZZ^n$ 
and for any $1 \leq i \leq k$ we have 
$\dim \Delta_i =n-k$. 
Then we have the concentration 
\begin{equation} 
H^j(W ; \LL ) \simeq 
0 \qquad (j \not= n-k). 
\end{equation} 
\end{theorem}

\begin{proof} 
Set $T=T_0 \times 
( \CC^*)^k_{t_1, \ldots, t_k} 
\simeq ( \CC^*)^n_{x,t}$ and consider the 
Laurent polynomials 
\begin{equation} 
\tl{P_i}(x,t)=t_i -P_i(x) 
\qquad (1 \leq i \leq k) 
\end{equation}
on $T$. For $1 \leq i \leq k$ we set also 
\begin{equation} 
Z_i = \{ (x,t) \in T 
\ | \ \tl{P_1}(x,t)= \cdots 
= \tl{P_i}(x,t)=0 \}. 
\end{equation}
We define a local system $\tl{\LL}$ on 
$T$ by 
\begin{equation} 
\tl{\LL} = 
\CC_T x_1^{c_1} 
\cdots x_{n-k}^{c_{n-k}} 
t_1^{\tl{c_1}} \cdots t_k^{\tl{c_k}}.
\end{equation}
Then $Z_k \simeq W$ 
and we have isomorphisms 
\begin{equation} 
H^j(W ; \LL ) \simeq 
H^j (Z_k ; \tl{\LL} ) 
 \qquad (j \in \ZZ ). 
\end{equation} 
First let us consider the case where 
$ \tl{c} = ( \tl{c_1}, 
\ldots, \tl{c_k}) \notin \ZZ^k$. 
In this case, without loss of generality 
we may assume that $\tl{c_k} \notin \ZZ$. 
Then by the K\"unneth formula, 
for $i=1,2, \ldots, k-1$ we have the 
vanishings 
\begin{equation}  
H^j (Z_i ; \tl{\LL} ) 
 \simeq 0 \qquad (j \in \ZZ ). 
\end{equation} 
Moreover we can naturally identify $Z_{k-1} 
\subset T$ with $(T_0 \setminus 
\cup_{i=1}^{k-1} P_i^{-1}(0)) \times 
\CC^*_{t_k}$. Consider $\tl{P_k}$ as a 
Laurent polynomial on $T_1=T_0 \times 
\CC^*_{t_k} \simeq ( \CC^*)^{n-k+1}$. 
Note that we have $\dim NP(\tl{P_k}) 
=n-k+1= \dim T_1$. 
By taking a sufficiently generic 
\begin{equation}  
( a_1, \ldots,  a_{n-k}, 
 a_{n-k+1} ) \in \Int  NP(\tl{P_k}) 
\subset \RR^{n-k+1} 
\end{equation} 
we define a real-valued function $g$ on 
$T_1=T_0 \times \CC^*_{t_k}$ by 
\begin{equation}  
g(x, t_k)= \left| 
\tl{P_k} (x, t_k) \times 
x_1^{- a_1} \cdots 
x_{n-k}^{- a_{n-k}} 
t_k^{- a_{n-k+1}}  \right|. 
\end{equation} 
Then by applying Lemma \ref{VL} to the 
Morse function $g: T_1=T_0 \times \CC^* 
 \longrightarrow \RR$ and arguing 
as the proof of Theorem \ref{NVTM} 
we obtain the desired concentration 
\begin{equation} 
H^j (Z_k ; \tl{\LL} ) 
\simeq 0 \qquad (j \not= n-k). 
\end{equation} 
The proof for the remaining case where 
$(c, \tl{c} )=
(c_1, \ldots, c_{n-k}, \tl{c_1}, 
\ldots, \tl{c_k}) \notin \ZZ^n$ and 
$ \tl{c} = ( \tl{c_1}, 
\ldots, \tl{c_k}) \in \ZZ^k$ 
is similar. In this case, $Z_1 \subset T$ 
is isomorphic to the product 
$Z_1^{\prime} \times ( \CC^*)^{k-1}$ 
for a hypersurface $Z_1^{\prime}$ in 
$T_0 \times \CC^*_{t_1}$ and $\tl{\LL}$ 
is isomoprhic to the pull-back of a 
local system on $T_0 
\times \CC^*_{t_1}$. 
Hence by the K\"unneth 
formula and the proof 
of Theorem \ref{NVTM} 
we obtain the concentration 
\begin{equation} 
H^j (Z_1 ; \tl{\LL} ) 
\simeq 0 \qquad (j \not= n-k, \ldots, n-1). 
\end{equation} 
Repeating this argument with the help of 
Lemma \ref{VL} and the proof 
of Theorem \ref{NVTM} we obtain also 
\begin{equation} 
H^j(W ; \LL ) \simeq 
H^j (Z_k ; \tl{\LL} ) 
\simeq 0 \qquad (j \not= n-k, \ldots, n-1). 
\end{equation} 
Then the assertion 
is obtained by applying 
Artin's vanishing theorem to 
the $(n-k)$-dimensional affine variety 
$Z_k \subset T$. This completes the proof. 
\end{proof}

\end{document}